\documentclass[11pt]{article}

\usepackage[T1]{fontenc}
\usepackage[english]{babel}
\usepackage{microtype}
\usepackage{amsmath,amssymb,amsfonts,amsthm,mathtools}
\usepackage{geometry}
\usepackage{graphicx}
\usepackage{hyperref}
\usepackage[shortlabels]{enumitem}
\usepackage{xcolor}

\newtheorem{theorem}{Theorem}[section]
\newtheorem{lemma}[theorem]{Lemma}
\newtheorem{claim}[theorem]{Claim}

\newtheorem{problem}[theorem]{Problem}

\newtheorem{definition}[theorem]{Definition}

\newcommand{\ceil}[1]{\lceil #1\rceil}
\newcommand{\floor}[1]{\lfloor #1\rfloor}

\title{Polylogarithmic Bounds for Nested Cycles without Geometric Crossings}
\author{
Yue Xu\textsuperscript{1}\thanks{Email: \href{mailto:xuyue1@sjtu.edu.cn}{xuyue1@sjtu.edu.cn}},\quad
Jiasheng Zeng\thanks{Email: \href{mailto:jasonzeng@mail.ustc.edu.cn}{jasonzeng@mail.ustc.edu.cn}},\quad
and Xiao-Dong Zhang\textsuperscript{1}\thanks{Email: \href{mailto:xiaodong@sjtu.edu.cn}{xiaodong@sjtu.edu.cn}}
\\[0.4em]
\textsuperscript{1}School of Mathematical Sciences, Shanghai Jiao Tong University, Shanghai, China
}
\date{\today}

\begin{document}

\begingroup
\renewcommand{\thefootnote}{\fnsymbol{footnote}}
\maketitle
\endgroup
\setcounter{footnote}{0}

\begin{abstract}
A problem of Erd\H{o}s asks for extremal conditions forcing edge-disjoint cycles with a prescribed nested structure.  In the geometric version, the nesting is required to be noncrossing with respect to the cyclic orders.  Fern\'andez, Kim, Kim and Liu proved that constant average degree forces two such cycles.  We prove a polylogarithmic bound for the natural multi-layer version: for every fixed $k\ge 3$, every sufficiently large $n$-vertex graph with at least
\[
        C_k n(\log n)^{k-1}(\log\log n)^{k-3}
\]
edges contains $k$ pairwise edge-disjoint nested cycles without geometric crossings.  The proof combines the robust sublinear expander framework of Alon, Buci\'c, Sauermann, Zakharov and Zamir with a controlled wrapping lemma that permits the layers to be built successively with controlled length.
\end{abstract}

\section{Introduction}

Extremal problems concerning cycles form a central theme in graph theory.  A classical result of Corr\'adi and Hajnal states that every graph $G$ with minimum degree at least $2k$ and $|G|\ge 3k$ contains $k$ vertex-disjoint cycles \cite{CorradiHajnal1963}.  This line of work was later developed in several directions.  H\"aggkvist and Egawa studied vertex-disjoint cycles of the same length \cite{Haggkvist1985,Egawa1996}.  More generally, cycles may be viewed as minimal subgraphs of minimum degree two or connectivity two, and there are corresponding decomposition and partition results for graphs under minimum degree or connectivity constraints~\cite{Hajnal1983,Thomassen1983,Stiebitz1996,KuhnOsthus2003}.  Verstra\"ete proved a related result on vertex-disjoint cycles whose lengths form an arithmetic progression \cite{Verstraete2002}.

In 1975, Erd\H{o}s asked several extremal questions about forcing edge-disjoint cycles with additional structure \cite{Erdos1975}.  The first one concerns nested cycles.  Cycles $C_1,\ldots,C_k$ in a graph are called nested if
\[
        V(C_k)\subseteq V(C_{k-1})\subseteq \cdots \subseteq V(C_1).
\]
If, in addition, their edge sets are pairwise disjoint, then they are edge-disjoint nested cycles.  Bollob\'as proved that a linear number of edges forces two edge-disjoint nested cycles \cite{Bollobas1978}.  He also asked whether an analogous linear bound holds for any fixed number of nested cycles.  This was confirmed by Chen, Erd\H{o}s and Staton, who proved that $O_k(n)$ edges force $k$ edge-disjoint nested cycles \cite{ChenErdosStaton1996}.

Erd\H{o}s also asked a stronger geometric version.  Suppose $V(C_2)\subseteq V(C_1)$.  If the vertices of $C_2$ are viewed in the cyclic order induced by $C_1$, then one may ask that the edges of $C_2$ do not cross inside $C_1$.  Equivalently, the cyclic order of $V(C_2)$ induced by $C_1$ agrees with the cyclic order of $C_2$.  This condition is substantially more rigid than ordinary nesting.  The proof of Chen, Erd\H{o}s and Staton proceeds by finding a cycle $C$ such that the subgraph induced on $V(C)$ remains dense, and then iterating this density increment.  However, this method gives no control over the cyclic order of the next cycle, and therefore does not address the no-crossing condition.

Fern\'andez, Kim, Kim and Liu \cite{FKKL2022} resolved the two-cycle geometric problem by proving that there is an absolute constant $C$ such that every graph of average degree at least $C$ contains two edge-disjoint nested cycles without geometric crossings.  Their proof uses sublinear expanders, a notion originating in work of Koml\'os and Szemer\'edi on topological cliques \cite{KomlosSzemeredi1996} and further developed by Haslegrave, Kim and Liu \cite{HaslegraveKimLiu2022}.  Sublinear expanders have played an important role in several recent results on sparse graph structure, including work on Hamiltonian subsets and clique subdivisions \cite{KimLiuSharifzadehStaden2017,LiuMontgomery2017}.  The main structural object in the proof of \cite{FKKL2022} is a kraken, which provides an inner cycle together with short arms that can be linked in cyclic order to obtain an outer noncrossing cycle.

The two-cycle theorem of \cite{FKKL2022} leaves open the natural multi-layer question.

\begin{problem}\label{prob:constant-degree}
For every fixed $k$, is there a constant $d_k$ such that every graph of average degree at least $d_k$ contains $k$ edge-disjoint nested cycles without geometric crossings?
\end{problem}

The constant-average-degree version appears to remain open already for $k=3$.  The main contribution of this paper is to give the first general upper bound for this multi-layer geometric problem: we show that a polylogarithmic average-degree assumption is sufficient for every fixed number of layers. For an integer $k\ge 2$, let $f_k(n)$ be the smallest integer such that every $n$-vertex graph with at least $f_k(n)$ edges contains pairwise edge-disjoint cycles $C_1,C_2,\ldots,C_k$ satisfying
\[
        V(C_k)\subseteq V(C_{k-1})\subseteq\cdots\subseteq V(C_1),
\]
and such that, for each $i=1,\ldots,k-1$, the cyclic order of $V(C_{i+1})$ induced by $C_i$ agrees with the cyclic order of $C_{i+1}$.  We call such cycles $k$ nested cycles without geometric crossings.

\begin{theorem}\label{thm:main}
For each fixed integer $k\ge 3$, there are constants $C_k>0$ and $n_k$ such that, for every $n\ge n_k$,
\[
        f_k(n)\le C_k n(\log n)^{k-1}(\log\log n)^{k-3}.
\]
\end{theorem}
Thus, while the constant-average-degree problem remains open for $k\ge 3$, Theorem~\ref{thm:main} shows that only a polylogarithmic average-degree assumption is sufficient for any fixed number of noncrossing layers.

Let us briefly describe the proof.  The main difficulty is that the wrapping operation has to be iterated.  It is not enough to find one noncrossing outer cycle around a given cycle; one must also keep the new cycle short enough so that the expansion available in the ambient graph can still be used at the next step.

We first pass to a robust sublinear expander, using the framework of Alon, Buci\'c, Sauermann, Zakharov and Zamir \cite{AlonBucicSauermannZakharovZamir2025}.  Starting from a shortest cycle $C_k$ in this expander, we repeatedly construct an outer cycle around the current one.  The key point is a controlled wrapping lemma: if the current cycle has length $\ell$, then, under the relevant robust-expansion scale condition, it can be wrapped by an outer cycle of length $O(\ell\log N\log\log N)$, with the prescribed cyclic order and without using edges inside the current cycle.  This last edge-avoidance condition is what allows the construction to be iterated while keeping all layers edge-disjoint.

After $k-2$ controlled wrapping steps we obtain $C_2$.  A final application of a linkedness theorem of Bollob\'as and Thomason gives the outermost cycle $C_1$.  The length control in the intermediate steps leads to the bound
\[
        C_k n(\log n)^{k-1}(\log\log n)^{k-3}.
\]

Let us also place the result among related problems.  Erd\H{o}s's third question asks for edge-disjoint cycles with the same vertex set.  This is stronger than ordinary nesting and different in nature from the no-crossing nested-cycle problem.  Pyber, R\"odl and Szemer\'edi gave lower-bound constructions related to the absence of regular subgraphs \cite{PyberRodlSzemeredi1995}.  Recently, Chakraborti, Janzer, Methuku and Montgomery \cite{ChakrabortiJanzerMethukuMontgomery2025} proved a polylogarithmic upper bound for finding edge-disjoint cycles with the same vertex set.  Their proof uses powerful reservoir and regularisation ideas.  Although that framework does not directly yield the ordered cyclic linkage needed here, it suggests a possible route toward further improvements.

There are also several related extremal questions about cycles with many additional edges or prescribed chord structure.  Chen, Erd\H{o}s and Staton asked about forcing a cycle with many chords \cite{ChenErdosStaton1996}; Dragani\'c, Methuku, Munh\'a Correia and Sudakov obtained a polylogarithmic improvement for cycles with many chords \cite{DraganicMethukuMunhaCorreiaSudakov2024}.  Brada\v{c}, Methuku and Sudakov resolved an old problem of Erd\H{o}s on cycles with all diagonals \cite{BradacMethukuSudakov2024}.  The related Erd\H{o}s--Sauer problem on regular subgraphs was resolved by Janzer and Sudakov, who showed that average degree $C_k\log\log n$ forces a $k$-regular subgraph \cite{JanzerSudakov2023}.  These results show that sparse graphs with polylogarithmic average degree already contain rich cyclic or regular structure, but the ordered noncrossing nesting condition considered here requires additional control over cyclic order.

\section{Preliminaries}

Throughout, all graphs are finite and simple.  All logarithms are natural.  We omit floors and ceilings whenever they do not affect the argument.  All constants are absolute unless otherwise stated.

\subsection{Basic notation}

For a positive integer $r$, we write $[r]:=\{1,2,\ldots,r\}$.  For a graph $G$, we write $V(G)$ and $E(G)$ for its vertex set and edge set, respectively.  The number of vertices and edges are denoted by $|G|:=|V(G)|$ and $e(G):=|E(G)|$, respectively.  The average degree of $G$ is denoted by
\[
        d(G):=\frac{2e(G)}{|G|},
\]
and the minimum degree of $G$ is denoted by $\delta(G)$.  For a vertex $v\in V(G)$, we write $d_G(v)$ for its degree in $G$.  If $X\subseteq V(G)$, then
\[
        N_G(X):=\{v\in V(G)\setminus X:\text{ there exists }x\in X\text{ with }vx\in E(G)\}
\]
denotes the external neighbourhood of $X$ in $G$.  For a single vertex $v$, we write $N_G(v):=N_G(\{v\})$.  If $F\subseteq E(G)$, then $G-F$ denotes the spanning subgraph obtained from $G$ by deleting the edges in $F$.  If $S\subseteq V(G)$, then $G-S$ denotes the induced subgraph $G[V(G)\setminus S]$.  For two vertex sets $A,B\subseteq V(G)$, the distance between $A$ and $B$ in $G$ is the minimum length of a path in $G$ with one endpoint in $A$ and the other endpoint in $B$.  If $A\cap B\ne\varnothing$, this minimum may be zero; throughout the paper a path of length zero is allowed and consists of a single vertex.  If no such path exists, this distance is taken to be infinite. The vertex-connectivity of a graph $G$, namely the minimum number of vertices whose removal disconnects $G$ or leaves a single vertex, is denoted by $\kappa(G)$.

\subsection{Nested cycles}

Let $C_1,\ldots,C_k$ be cycles in a graph $G$.  We say that they are $k$ nested cycles without geometric crossings if
\[
        V(C_k)\subseteq V(C_{k-1})\subseteq\cdots\subseteq V(C_1),
\]
the cycles are pairwise edge-disjoint, and, for every $i=1,\ldots,k-1$, the cyclic order of $V(C_{i+1})$ induced by $C_i$ agrees with the cyclic order of $C_{i+1}$.  Equivalently, if the vertices of $V(C_{i+1})$ are read in their cyclic order along $C_i$, then this order agrees with their cyclic order along $C_{i+1}$, up to reversal.

\subsection{Robust sublinear expanders}

We use the following robust expansion notion.

\begin{definition}[Robust sublinear expander]\label{def:robust-expander}
A graph $H$ on $N\ge 2$ vertices is called a robust sublinear expander if, for every $0\le \alpha\le 1$ and every non-empty set $U\subseteq V(H)$ with
\[
        |U|\le N^{1-\alpha},
\]
and for every set $F\subseteq E(H)$ with
\[
        |F|\le \frac{\alpha}{3}d(H)|U|,
\]
one has
\[
        |N_{H-F}(U)|\ge \frac{\alpha}{3}|U|.
\]
\end{definition}

We need the following expander-reduction lemma.  The robust-expansion part is used exactly as a black box from Lemma~3.2 of \cite[Lemma~3.2, p.~9]{AlonBucicSauermannZakharovZamir2025}.  Our Definition~\ref{def:robust-expander} is chosen to match that lemma: the edge-deletion allowance is $(\alpha/3)d(H)|U|$, and the guaranteed external neighbourhood size is $(\alpha/3)|U|$.  Thus no change of constants is hidden in the invocation below.  We include only the short minimum-degree pruning argument, since it is a convenient consequence of taking the same extremal subgraph.

\begin{lemma}[Robust expander reduction]\label{lem:robust-reduction}
Let $G$ be a graph with at least one edge.  Then $G$ contains a subgraph $H$ such that $H$ is a robust sublinear expander in the sense of Definition~\ref{def:robust-expander} and
\[
        d(H)\ge \frac13\frac{\log |V(H)|}{\log |V(G)|}d(G).
\]
Moreover, $H$ can be chosen so that
\[
        \delta(H)\ge \frac12 d(H).
\]
\end{lemma}

\begin{proof}
Choose a subgraph $H\subseteq G$ as follows. Among all subgraphs $J\subseteq G$ with $V(J)\neq\emptyset$. We choose H to be a subgraph of G that maximizes
\[
        \frac{d(J)}{\log |V(J)|-1/3}.
\]
Since $G$ has at least one edge, this maximum is positive. Hence $d(H)>0$, and in particular $H$ has at least one edge and $|V(H)|\ge 2$. By the proof of \cite[Lemma~3.2, pp.~9--11]{AlonBucicSauermannZakharovZamir2025}, with $\varepsilon$ there renamed as $\alpha$, this extremal subgraph is a robust sublinear expander in the sense of Definition~\ref{def:robust-expander}. Moreover, the maximality of $H$, applied to $J=G$, gives
\[
        d(H)\ge
        \frac{\log |V(H)|-1/3}{\log |V(G)|-1/3}\,d(G).
\]
Since $|V(H)|\ge 2$, we have $\log |V(H)|-1/3\ge \frac13\log |V(H)|$, and also $\log |V(G)|-1/3\le \log |V(G)|$. Therefore
\[
        d(H)\ge
        \frac13\,\frac{\log |V(H)|}{\log |V(G)|}\,d(G),
\]
which is the asserted lower bound.

It remains to record the minimum-degree consequence.  Suppose that some vertex $v\in V(H)$ satisfies $d_H(v)<d(H)/2$.  Let $m=|V(H)|$ and put $H'=H-v$.  We claim that $H'$ has at least one edge.  Indeed, if $H'$ were edgeless, then every edge of $H$ would be incident with $v$, so
\[
        d_H(v)=e(H)=\frac{m d(H)}2>\frac{d(H)}2,
\]
contrary to the choice of $v$.  Thus $H'$ has at least one edge, and in particular $|V(H')|\ge 2$ and the denominator below is positive.  Moreover,
\[
        d(H')=\frac{m d(H)-2d_H(v)}{m-1}>d(H).
\]
Also $\log |V(H')|-1/3<\log |V(H)|-1/3$.  Hence
\[
        \frac{d(H')}{\log |V(H')|-1/3}>
        \frac{d(H)}{\log |V(H)|-1/3},
\]
contradicting the choice of $H$.  Thus $\delta(H)\ge d(H)/2$.
\end{proof}

\subsection{Linkedness and the Moore bound}

We use the following theorem of Bollob\'as and Thomason.

\begin{theorem}[Bollob\'as--Thomason linkedness theorem~\cite{BollobasThomason1996}]\label{thm:linkedness}
There is an absolute constant $K_{\mathrm{link}}$ such that every $K_{\mathrm{link}}t$-connected graph is $t$-linked.  That is, for any distinct vertices
\[
        p_1,q_1,\ldots,p_t,q_t,
\]
there exist pairwise vertex-disjoint paths $P_i$ joining $p_i$ to $q_i$, for $i\in[t]$.
\end{theorem}

We also use the following standard Moore bound.

\begin{lemma}[Moore bound]\label{lem:moore}
Let $G$ be an $N$-vertex graph with minimum degree at least $\delta\ge 3$.  Then $G$ contains a cycle of length at most
\[
        A_{\mathrm M}\frac{\log N}{\log \delta}
\]
for an absolute constant $A_{\mathrm M}$.
\end{lemma}

\begin{proof}
Let $g$ be the girth of $G$.  For a vertex $v$, let $B_G(v,r)$ be the subgraph induced by all vertices at distance at most $r$ from $v$.  If $g>2r+1$, then $B_G(v,r)$ is a tree for every $v\in V(G)$.  Hence it has at least
\[
        1+\delta\sum_{i=0}^{r-1}(\delta-1)^i
\]
vertices.  Choosing $r=\floor{\log_{\delta-1}N}+1$ gives more than $N$ vertices, a contradiction.  Thus
\[
        g\le 2\log_{\delta-1}N+2.
\]
The stated estimate follows by changing the absolute constant.
\end{proof}

\section{Robust short connections}

The following lemma is the main quantitative input.  It converts robust sublinear expansion into vertex-avoiding connections of length $O(\log N\log\log N)$.

\begin{lemma}[Robust short connection]\label{lem:robust-short-connection}
There exist absolute constants $A_{\mathrm{sc}}>0$ and $N_{\mathrm{sc}}$ such that the following holds.  Let $H$ be an $N$-vertex robust sublinear expander in the sense of Definition~\ref{def:robust-expander} with $N\ge N_{\mathrm{sc}}$.  Let $\sigma$ be a real number satisfying
\[
        1\le \sigma\le N^{1/2}.
\]
Let $S\subseteq V(H)$ satisfy
\[
        |S|\le \frac{\sigma}{1000},
\]
and let $X,Y\subseteq V(H)\setminus S$ satisfy
\[
        |X|,|Y|\ge \frac{\sigma}{4}.
\]
Then $X$ and $Y$ are joined in $H-S$ by a path of length at most
\[
        A_{\mathrm{sc}}\log N\log\log N.
\]
\end{lemma}

\begin{proof}
We prove the statement with $A_{\mathrm{sc}}=40$, after increasing $N_{\mathrm{sc}}$ if necessary.

For a set $U\subseteq V(H)$ with $1\le |U|\le N/2$, define
\[
        \alpha(U):=\frac{\log(N/|U|)}{\log N}.
\]
Then $|U|=N^{1-\alpha(U)}$ and $0<\alpha(U)\le 1$.  Applying Definition~\ref{def:robust-expander} with $F=\varnothing$ gives
\[
        |N_H(U)|\ge \frac{\alpha(U)}{3}|U|.
\]
Consequently, for every such $U\subseteq V(H)\setminus S$,
\[
        |N_{H-S}(U)|
        =|N_H(U)\setminus S|
        \ge |N_H(U)|-|S|
        \ge \frac{\alpha(U)}{3}|U|-|S|.
\]

We first record a simple estimate.  For every sufficiently large $N$, every $1\le \sigma\le N^{1/2}$, and every real $u$ satisfying
\[
        \max\{1,\sigma/4\}\le u\le N/2,
\]
one has \begin{equation}\label{ineq:sigma/40}
     u\frac{\log(N/u)}{\log N}\ge \frac{\sigma}{40}.
\end{equation}

Now grow a ball from $X$ inside $H-S$.  Set $B_0:=X$ and define recursively
\[
        B_{t+1}:=B_t\cup N_{H-S}(B_t).
\]
Thus $B_t$ is the set of vertices reachable from $X$ in $H-S$ by a path of length at most $t$. Since $B_t\supseteq X$, we have $|B_t|\ge |X|\ge \max\{1, \sigma/4\}$.  Put $ z_t:=\log\frac{N}{|B_t|}$ and 
\[
        \alpha_t:=\alpha(B_t)=\frac{\log(N/|B_t|)}{\log N}=\frac{z_t}{\log N}.
\]
\begin{claim}\label{claim:B_t<N/2 then B_t1 grows}
    If $|B_t|\le N/2$, then $ |B_{t+1}|\ge \left(1+\frac{\alpha_t}{4}\right)|B_t|.$
\end{claim}
\begin{proof}[Proof of Claim~\ref{claim:B_t<N/2 then B_t1 grows}]
    By the inequality~\ref{ineq:sigma/40}, $\alpha_t |B_t|\ge \frac{\sigma}{40}.$ Since $|S|\le \sigma/1000$, we get 
    \[|S|\le \frac1{25}\alpha_t |B_t|. \]Therefore
\[
        |N_{H-S}(B_t)|
        \ge \frac{\alpha_t}{3}|B_t|-|S|
        \ge \left(\frac13-\frac1{25}\right)\alpha_t |B_t|
        \ge \frac{\alpha_t}{4}|B_t|.
\]
Hence, whenever $|B_t|\le N/2$, we have $|B_{t+1}|\ge \left(1+\frac{\alpha_t}{4}\right)|B_t|. $
\end{proof}
Take
\[
        T:=\ceil{16\log N\log\log N}.
\]
\begin{claim}\label{claim:BT>N/2}
    $|B_T|>N/2$.
\end{claim}
\begin{proof}[Proof of Claim~\ref{claim:BT>N/2}]
    If $|B_s|>N/2$ for some $s\le T$, then the claim is immediate, since the sets $B_t$ are increasing.  Otherwise $|B_t|\le N/2$ for every $0\le t\le T$, and the following estimates apply for every $0\le t<T$.

As long as $|B_t|\le N/2$, we have $0<z_t\le \log N$ and $\alpha_t=z_t/\log N$.  The growth estimate gives
\[
        z_{t+1}\le z_t-\log\left(1+\frac{z_t}{4\log N}\right).
\]
Since $0\le z_t/(4\log N)\le 1/4$, we have $\log(1+w)\ge w/2$ for $0\le w\le 1/4$.  Thus
\[
        z_{t+1}\le \left(1-\frac1{8\log N}\right)z_t.
\]
Iterating up to time $T$ gives
\[
        z_T\le \exp\left(-\frac{T}{8\log N}\right)\log N.
\]
By the choice of $T$, and after increasing $N_{\mathrm{sc}}$ if necessary, the right-hand side is smaller than $\log 2$.  Hence $|B_T|>N/2$, proving the claim.
\end{proof}
Applying the same argument to $Y$, we obtain a set $B'_T$ of vertices reachable from $Y$ in $H-S$ within distance at most $T$, with $|B'_T|>N/2$.  Thus $B_T\cap B'_T\ne\varnothing$.  Concatenating a path from $X$ to a vertex in the intersection with a path from that vertex to $Y$, and then deleting repeated vertices if necessary, gives a path in $H-S$ from $X$ to $Y$ of length at most $2T\le 40\log N\log\log N$ for all sufficiently large $N$.
\end{proof}

We shall also need a connectivity consequence of the same expansion.

\begin{lemma}[Robust connectivity]\label{lem:robust-connectivity}
There exist absolute constants $c_{\mathrm{con}}>0$ and $N_{\mathrm{con}}$ such that the following holds.  Let $H$ be an $N$-vertex robust sublinear expander with $N\ge N_{\mathrm{con}}$ and $\delta(H)\ge q$.  Put
\[
        \sigma:=\min\{q,N^{1/2}\}.
\]
Then
\[
        \kappa(H)\ge c_{\mathrm{con}}\sigma.
\]
\end{lemma}

\begin{proof}
Choose $c_{\mathrm{con}}>0$ sufficiently small.  Suppose, for a contradiction, that $S\subseteq V(H)$ is a vertex cut with $|S|<c_{\mathrm{con}}\sigma.$ Recall that, by our convention, a vertex cut is allowed to leave a single vertex.  We first dispose of this case.  If $H-S$ has exactly one vertex, then
\[
        |S|=N-1\ge \frac12 N^{1/2}\ge \frac12\sigma
\]
for $N\ge 4$, contradicting $|S|<c_{\mathrm{con}}\sigma$ once $c_{\mathrm{con}}<1/2$.  Hence $H-S$ is disconnected and has at least two components.

If $A$ is a component and $v\in A$, then all neighbours of $v$ lie in $A\cup S$, so
\[
        q\le d_H(v)\le |A|-1+|S|.
\]
Thus, every component of $H-S$ has size at least $q-|S|+1\ge q/2$, provided $c_{\mathrm{con}}\le 1/4$. Therefore, there is a component $U$ of $H-S$ such that
\[
        \frac{q}{2}\le |U|\le \frac N2.
\]
Indeed, since $H-S$ has at least two components, at least one component has size at most $N/2$.  Every component has size at least $q/2$, by the previous paragraph.  Taking such a component gives the desired $U$. Since $U$ is a component of $H-S$, we have $N_H(U)\subseteq S.$ Let
\[
        \alpha:=\frac{\log(N/|U|)}{\log N}.
\]
By robust expansion with $F=\varnothing$,
\[
        |S|\ge |N_H(U)|\ge \frac{\alpha}{3}|U|.
\]
Since $q\ge\sigma$, the lower bound $|U|\ge q/2$ implies $|U|\ge\sigma/2$.  The same endpoint estimate used in the proof of Lemma~\ref{lem:robust-short-connection}, applied with lower endpoint $\max\{1,\sigma/2\}$, gives
\[
        \alpha |U|=|U|\frac{\log(N/|U|)}{\log N}\ge c_0\sigma
\]
for some absolute constant $c_0>0$, provided $N$ is sufficiently large.  Hence
\[
        |S|\ge \frac{c_0}{3}\sigma,
\]
contradicting the choice of $c_{\mathrm{con}}<c_0/3$.
\end{proof}

\section{Wrapping lemmas}

The following two lemmas are stated with an edge-disjointness conclusion strong enough to be iterated.  The point is to avoid not only the current cycle, but every edge whose two endpoints lie inside the current cycle.

\begin{figure}[t]
    \centering
    \includegraphics[width=0.62\textwidth]{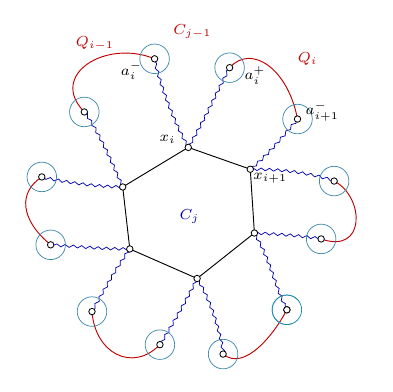}
    \caption{
    Controlled wrapping in one inductive step.  In the iteration, the current cycle is denoted by $C_j$ and the newly constructed outer cycle by $C_{j-1}$.  Each vertex $x_i$ of $C_j$ is assigned two distinct external neighbours $a_i^-$ and $a_i^+$. The paths $Q_i$ join $a_i^+$ to $a_{i+1}^-$ outside $V(C_j)$ and are chosen internally disjoint.  Together with the edges $x_i a_i^-$ and $x_i a_i^+$, these paths form the new outer cycle $C_{j-1}$, which visits the vertices of $C_j$ in their original
    cyclic order.
    }
    \label{fig:controlled-wrapping}
\end{figure}

\begin{lemma}[Controlled wrapping]\label{lem:controlled-wrapping}
Let $H$ be an $N$-vertex robust sublinear expander with $N\ge N_{\mathrm{sc}}$ and $\delta(H)\ge q$.  Put
\[
        \sigma:=\min\{q,N^{1/2}\},
        \qquad
        R:=A_{\mathrm{sc}}\log N\log\log N.
\]
Let
\[
        C=x_1x_2\cdots x_\ell x_1
\]
be a cycle in $H$.  There exists an absolute constant $B_{\mathrm{cw}}>0$ such that, if
\[
        \sigma\ge B_{\mathrm{cw}}\ell R,
\]
then $H$ contains a cycle $C^+$ satisfying
\[
        V(C)\subseteq V(C^+),
        \qquad
        E(C^+)\cap E(H[V(C)])=\varnothing,
\]
the cyclic order induced by $C^+$ on $V(C)$ agrees with the cyclic order of $C$, and
\[
        |C^+|\le B_{\mathrm{cw}}\ell R.
\]
\end{lemma}

\begin{proof}
Choose $B_{\mathrm{cw}}$ sufficiently large.  Since $R\ge 1$ and $\sigma\ge B_{\mathrm{cw}}\ell R$, we may assume
\[
        \sigma\ge 100\ell
        \qquad\text{and}\qquad
        6\ell R\le \frac{\sigma}{1000}.
\]
We can greedily choose distinct external terminals.  Indeed, when the two terminals at $x_i$ are chosen, at most $\ell-1$ neighbours of $x_i$ lie in $V(C)$ and fewer than $2\ell$ previously chosen terminals are forbidden.  Since $q\ge\sigma\ge 100\ell$, at least $q-3\ell\ge 2$ neighbours remain available.  Thus we may choose distinct vertices
\[
        a_i^-,a_i^+\in N_H(x_i)\setminus V(C),
        \qquad i\in[\ell].
\]

We next construct paths
\[
        Q_i:a_i^+\leadsto a_{i+1}^-,
        \qquad i\in[\ell],
\]
where indices are taken modulo $\ell$, such that they are pairwise internally vertex-disjoint and avoid $V(C)$ and all terminals internally.

For a path $P$, let $\operatorname{Int}(P)$ denote its set of internal vertices.  Suppose that $Q_1,\ldots,Q_{i-1}$ have already been constructed, each with length at most $R+2$, and satisfying the required disjointness conditions.  Let
\[
        S_i:=V(C)\cup\{a_r^-,a_r^+:r\in[\ell]\}\cup
        \bigcup_{r<i}\operatorname{Int}(Q_r).
\]
Then
\[
        |S_i|\le 3\ell+(i-1)(R+1)\le 6\ell R\le \frac{\sigma}{1000}.
\]
Moreover,
\[
        |N_H(a_i^+)\setminus S_i|\ge q-|S_i|\ge \sigma-\frac{\sigma}{1000}\ge \frac{\sigma}{4},
\]
and similarly
\[
        |N_H(a_{i+1}^-)\setminus S_i|\ge \frac{\sigma}{4}.
\]
Choose subsets
\[
        X_i\subseteq N_H(a_i^+)\setminus S_i,
        \qquad
        Y_i\subseteq N_H(a_{i+1}^-)\setminus S_i
\]
with
\[
        |X_i|=|Y_i|=\ceil{\sigma/4}.
\]
By Lemma~\ref{lem:robust-short-connection}, the sets $X_i$ and $Y_i$ are joined in $H-S_i$ by a path $P_i$ of length at most $R$.  Let $u_i\in X_i$ and $v_i\in Y_i$ be the endpoints of $P_i$, possibly with $u_i=v_i$.  Define $Q_i$ by adding the edge $a_i^+u_i$ to the beginning of $P_i$ and the edge $v_i a_{i+1}^-$ to the end.  Then $Q_i$ is an $a_i^+$--$a_{i+1}^-$ path of length at most $R+2$.  Since $P_i$ lies in $H-S_i$ and $X_i,Y_i\subseteq V(H)\setminus S_i$, this path has no internal vertex in $V(C)$, contains no terminal as an internal vertex, and is internally disjoint from all previously constructed paths.

After constructing all $Q_i$, define
\[
\begin{aligned}
        C^+:={}&x_1a_1^+Q_1a_2^-x_2a_2^+Q_2a_3^-x_3
        \cdots \\
        &\cdots x_\ell a_\ell^+Q_\ell a_1^-x_1.
\end{aligned}
\]
The terminals are all distinct, and the paths $Q_i$ are pairwise internally vertex-disjoint and contain no terminal as an internal vertex.  Hence $C^+$ is a cycle.  It passes through $x_1,x_2,\ldots,x_\ell$ in the same cyclic order as $C$.

Finally, every edge of $C^+$ either has one endpoint in $V(C)$ and the other outside $V(C)$, or lies entirely outside $V(C)$.  Hence
\[
        E(C^+)\cap E(H[V(C)])=\varnothing.
\]
Moreover,
\[
        |C^+|\le 3\ell+\ell(R+1)\le B_{\mathrm{cw}}\ell R
\]
after increasing $B_{\mathrm{cw}}$ if necessary.
\end{proof}

\begin{lemma}[Linked wrapping]\label{lem:linked-wrapping}
Let $H$ be an $N$-vertex robust sublinear expander with $N\ge N_{\mathrm{con}}$ and $\delta(H)\ge q$.  Put
\[
        \sigma:=\min\{q,N^{1/2}\}.
\]
Let $C=x_1x_2\cdots x_\ell x_1$ be a cycle in $H$.  There exists an absolute constant $B_{\mathrm{lw}}>0$ such that, if
\[
        \sigma\ge B_{\mathrm{lw}}\ell,
\]
then $H$ contains a cycle $C^+$ satisfying
\[
        V(C)\subseteq V(C^+),
        \qquad
        E(C^+)\cap E(H[V(C)])=\varnothing,
\]
and the cyclic order induced by $C^+$ on $V(C)$ agrees with the cyclic order of $C$.
\end{lemma}

\begin{proof}
By Lemma~\ref{lem:robust-connectivity},
\[
        \kappa(H)\ge c_{\mathrm{con}}\sigma.
\]
Choose $B_{\mathrm{lw}}$ large enough so that $B_{\mathrm{lw}}\ge 4$ and
\[
        c_{\mathrm{con}}\sigma-\ell\ge K_{\mathrm{link}}\ell
\]
whenever $\sigma\ge B_{\mathrm{lw}}\ell$.

Since $q\ge \sigma\ge B_{\mathrm{lw}}\ell$, the terminals can be chosen greedily.  More precisely, after avoiding $V(C)$ and the previously chosen terminals, each $x_i$ still has at least
\[
        q-(\ell-1)-2\ell\ge 2
\]
available neighbours.  Hence we may choose distinct terminals
\[
        a_i^-,a_i^+\in N_H(x_i)\setminus V(C),
        \qquad i\in[\ell].
\]
In particular, $H-V(C)$ contains at least these $2\ell$ vertices.

We use the elementary fact that, for every graph $G$ and every vertex set $Z$ such that $G-Z$ has at least two vertices,
\[
        \kappa(G-Z)\ge \kappa(G)-|Z|.
\]
Indeed, if some set $T\subseteq V(G)\setminus Z$ with $|T|<\kappa(G)-|Z|$ disconnected $G-Z$ or left it with a single vertex, then $Z\cup T$ would disconnect $G$ or leave it with a single vertex, while
\[
        |Z\cup T|<\kappa(G),
\]
a contradiction.

Applying this with $G=H$ and $Z=V(C)$ gives
\[
        \kappa(H-V(C))\ge \kappa(H)-|V(C)|
        \ge c_{\mathrm{con}}\sigma-\ell
        \ge K_{\mathrm{link}}\ell.
\]
Thus $H-V(C)$ is $\ell$-linked by Theorem~\ref{thm:linkedness}.  Therefore there are pairwise vertex-disjoint paths
\[
        P_i:a_i^+\leadsto a_{i+1}^-,
        \qquad i\in[\ell],
\]
where indices are taken modulo $\ell$.

The concatenation
\[
\begin{aligned}
        C^+:={}&x_1a_1^+P_1a_2^-x_2a_2^+P_2a_3^-x_3
        \cdots \\
        &\cdots x_\ell a_\ell^+P_\ell a_1^-x_1
\end{aligned}
\]
is a cycle, and it visits $x_1,x_2,\ldots,x_\ell$ in their original cyclic order.  Since all paths $P_i$ lie in $H-V(C)$, every edge of $C^+$ either joins a vertex of $C$ to a vertex outside $C$, or lies outside $V(C)$.  Therefore
\[
        E(C^+)\cap E(H[V(C)])=\varnothing.
\]
\end{proof}

\section{Scale transfer and proof of the main theorem}

We need one elementary lemma ensuring that the polylogarithmic scale survives the passage from the original graph to the robust expander subgraph.

\begin{lemma}[Transfer of the polylogarithmic scale]\label{lem:scale-transfer}
Fix $k\ge 3$, constants $B,c>0$, and an integer $N_{\min}\ge 2$.  There exist constants $C=C(k,B,c,N_{\min})>0$ and $n_0=n_0(k,B,c,N_{\min})$ such that the following holds for all $n\ge n_0$.  Let $2\le N\le n$, and put
\[
        L:=\log n,
        \qquad
        \Lambda:=\log\log n,
        \qquad
        L_N:=\log N,
        \qquad
        \Lambda_N:=\log\log N.
\]
Let $q$ be a real number satisfying $q\le N$ and
\[
        q\ge c\frac{L_N}{L} C L^{k-1}\Lambda^{k-3}.
\]
Put
\[
        \sigma:=\min\{q,N^{1/2}\}.
\]
Then $N\ge N_{\min}$ and
\[
        \sigma\ge B L_N^{k-1}\Lambda_N^{k-3}.
\]
\end{lemma}

\begin{proof}
We choose $C$ large enough.  First choose $N_1=N_1(k,B,N_{\min})$ so large that
\[
        N_1\ge \max\{N_{\min},e^e\}
\]
and
\[
        N^{1/2}\ge B(\log N)^{k-1}(\log\log N)^{k-3}
\]
for every $N\ge N_1$.

We may assume $N\ge N_1$.  Indeed, if $2\le N<N_1$, then $q\le N<N_1$, while
\[
        q\ge cC L_N L^{k-2}\Lambda^{k-3}\ge cC(\log 2)L^{k-2}\Lambda^{k-3},
\]
which is larger than $N_1$ for all sufficiently large $n$.  Thus this case is impossible after increasing $n_0$.  Consequently $N\ge N_1\ge N_{\min}$.

Now $N\ge N_1\ge e^e$, so $L_N\le L$ and $\Lambda_N\le\Lambda$.  If $q\le N^{1/2}$, then $\sigma=q$ and
\[
        \sigma=q\ge cC L_N L^{k-2}\Lambda^{k-3}
        \ge cC L_N^{k-1}\Lambda_N^{k-3}.
\]
Taking $C\ge B/c$ gives the desired conclusion.

If $q>N^{1/2}$, then $\sigma=N^{1/2}$, and the choice of $N_1$ gives directly
\[
        \sigma=N^{1/2}\ge B L_N^{k-1}\Lambda_N^{k-3}.
\]
The lemma follows.
\end{proof}

\begin{proof}[Proof of Theorem~\ref{thm:main}]
Fix $k\ge 3$, and let $n$ be sufficiently large.  All lower bounds on $n$ below depend only on $k$ and on the absolute constants appearing in the preceding lemmas.

Let
\[
        L:=\log n,
        \qquad
        \Lambda:=\log\log n.
\]
Let $B_*=B_*(k)\ge 1$ be a sufficiently large constant dominating the constants required in Lemmas~\ref{lem:controlled-wrapping} and~\ref{lem:linked-wrapping}, and also large enough for the inductive estimates below.  Put
\[
        N_{\min}:=\max\{N_{\mathrm{sc}},N_{\mathrm{con}},\ceil{e^e}\}.
\]
Choose $C_*=C_*(k,B_*,1/6,N_{\min})$ sufficiently large, as in Lemma~\ref{lem:scale-transfer}, and let
\[
        D:=C_* L^{k-1}\Lambda^{k-3}.
\]
Let $G$ be an $n$-vertex graph with average degree at least $D$.  By Lemma~\ref{lem:robust-reduction}, $G$ contains a robust sublinear expander $H$ on $N$ vertices such that
\[
        d(H)\ge \frac13\frac{\log N}{\log n}D
\]
and
\[
        \delta(H)\ge \frac12d(H).
\]
Put
\[
        q:=\delta(H),
        \qquad
        \sigma:=\min\{q,N^{1/2}\},
        \qquad
        L_N:=\log N,
        \qquad
        \Lambda_N:=\log\log N.
\]
Applying Lemma~\ref{lem:scale-transfer} with $B=B_*$, $c=1/6$ and the above value of $N_{\min}$ gives
\[
        N\ge N_{\min}
        \qquad\text{and}\qquad
        \sigma\ge B_* L_N^{k-1}\Lambda_N^{k-3}.
\]
In particular, $N$ is sufficiently large for Lemmas~\ref{lem:robust-short-connection} and~\ref{lem:robust-connectivity} to apply.  Also $q\ge \sigma\ge 3$, after increasing $B_*$ if necessary.

By the Moore bound, $H$ contains a cycle $C_k$ satisfying
\[
        |C_k|\le A_{\mathrm M}\frac{L_N}{\log q}.
\]
Put
\[
        R:=A_{\mathrm{sc}}L_N\Lambda_N.
\]
We next construct $C_{k-1},C_{k-2},\ldots,C_2$ by iterated controlled wrapping.  The length bound used in the induction is the explicit finite-iteration estimate
\[
        |C_j|\le A_{\mathrm M}B_{\mathrm{cw}}^{k-j}\frac{L_N R^{k-j}}{\log q}
        \le A'_k\frac{L_N R^{k-j}}{\log q},
        \qquad 2\le j\le k,
\]
where $A'_k:=A_{\mathrm M}\max\{1,B_{\mathrm{cw}}^k\}$.  For $j=k$, this is precisely the Moore-bound estimate.

Suppose $3\le j\le k$ and $C_j$ has already been constructed with the displayed bound.  To apply Lemma~\ref{lem:controlled-wrapping}, it is enough to check
\[
        \sigma\ge B_{\mathrm{cw}} |C_j|R.
\]
Using the finite-iteration estimate, we have
\[
        B_{\mathrm{cw}} |C_j|R
        \le A_{\mathrm M}B_{\mathrm{cw}}^{k-j+1}\frac{L_N R^{k-j+1}}{\log q}
        \le A'_k B_{\mathrm{cw}}\frac{L_N R^{k-2}}{\log q},
\]
because $j\ge 3$ and $R\ge 1$.  Since
\[
        R=A_{\mathrm{sc}}L_N\Lambda_N,
\]
the right-hand side is at most
\[
        A''_k\frac{L_N^{k-1}\Lambda_N^{k-2}}{\log q}.
\]
The lower bound on $\sigma$ gives
\[
        q\ge \sigma\ge B_* L_N^{k-1}\Lambda_N^{k-3}.
\]
Since $B_*\ge 1$ and $N\ge e^e$, we have $\Lambda_N\ge 1$, and hence
\[
        \log q
        \ge \log\!\left(B_* L_N^{k-1}\Lambda_N^{k-3}\right)
        \ge (k-1)\Lambda_N
        \ge \Lambda_N.
\]
Hence
\[
        B_{\mathrm{cw}} |C_j|R
        \le A'''_k L_N^{k-1}\Lambda_N^{k-3}.
\]
Choosing $B_*=B_*(k)$ large enough in terms of $k$ and the constants above yields
\[
        \sigma\ge B_{\mathrm{cw}} |C_j|R.
\]
Thus Lemma~\ref{lem:controlled-wrapping} produces a cycle $C_{j-1}$ such that
\[
        V(C_j)\subseteq V(C_{j-1}),
        \qquad
        E(C_{j-1})\cap E(H[V(C_j)])=\varnothing,
\]
the cyclic order induced by $C_{j-1}$ on $V(C_j)$ agrees with the cyclic order of $C_j$, and
\[
        |C_{j-1}|\le B_{\mathrm{cw}} |C_j|R
        \le A_{\mathrm M}B_{\mathrm{cw}}^{k-j+1}\frac{L_N R^{k-j+1}}{\log q}.
\]
This is exactly the displayed finite-iteration bound with $j$ replaced by $j-1$.  The induction is complete down to $C_2$.

In particular,
\[
        |C_2|\le A_{\mathrm M}B_{\mathrm{cw}}^{k-2}\frac{L_N R^{k-2}}{\log q}
        \le A'_k L_N^{k-1}\Lambda_N^{k-3}.
\]
Again choosing $B_*$ sufficiently large gives
\[
        \sigma\ge B_{\mathrm{lw}}|C_2|.
\]
By Lemma~\ref{lem:linked-wrapping}, there is a cycle $C_1$ such that
\[
        V(C_2)\subseteq V(C_1),
        \qquad
        E(C_1)\cap E(H[V(C_2)])=\varnothing,
\]
and the cyclic order induced by $C_1$ on $V(C_2)$ agrees with the cyclic order of $C_2$.

It remains to check pairwise edge-disjointness.  We prove a stronger invariant.  Whenever $C_{j-1}$ is constructed from $C_j$, the wrapping lemma gives
\[
        E(C_{j-1})\cap E(H[V(C_j)])=\varnothing.
\]
For every $i\ge j$, we have $V(C_i)\subseteq V(C_j)$, and hence
\[
        E(C_i)\subseteq E(H[V(C_j)]).
\]
Therefore $C_{j-1}$ is edge-disjoint from every inner cycle $C_i$ with $i\ge j$.  Applying this at each construction step shows that
\[
        C_1,C_2,\ldots,C_k
\]
are pairwise edge-disjoint.

The vertex containments
\[
        V(C_k)\subseteq V(C_{k-1})\subseteq\cdots\subseteq V(C_1)
\]
and the preservation of cyclic order at each wrapping step show that the cycles are nested without geometric crossings.  Thus every $n$-vertex graph with average degree at least $C_*(\log n)^{k-1}(\log\log n)^{k-3}$ contains the desired configuration.  If an $n$-vertex graph has at least
\[
        \frac{C_*}{2}n(\log n)^{k-1}(\log\log n)^{k-3}
\]
edges, then its average degree is at least $C_*(\log n)^{k-1}(\log\log n)^{k-3}$.  Therefore the factor $1/2$ is absorbed by redefining the constant $C_k$, and the asserted bound on $f_k(n)$ follows for all sufficiently large $n$. 
\end{proof}






\section*{Acknowledgment}

The work is partly  supported by National Natural Science Foundation of China (Nos. 12371354, W2521102), the Science and Technology Commission of Shanghai Municipality (No.25LN3200600) and the Montenegrin-Chinese Science and Technology Cooperation Project (No. 4-3).

\bibliographystyle{alpha}
\bibliography{new_refs}

\end{document}